\documentclass[ oneside, final]{amsart}

\usepackage{ifdraft}

\usepackage[colorlinks=true, citecolor=green!50!black, linkcolor=red!50!black, final]{hyperref}
\usepackage[style=alphabetic, backend=biber, backref=true, bibencoding=utf8]{biblatex}

\ifdraft{\usepackage[notref, notcite]{showkeys}}{}

\usepackage[letterpaper]{geometry}
\usepackage{xargs}
\usepackage{amsmath, amsthm, mathrsfs, amssymb}
\usepackage{xypic}
\usepackage{graphicx}
\usepackage[usenames, dvipsnames, svgnames, table]{xcolor}
\usepackage{pdfsync}
\usepackage{bm}
\usepackage{enumitem}

\makeatletter
\@namedef{subjclassname@2020}{%
  \textup{2020} Mathematics Subject Classification}
\makeatother

\usepackage[colorinlistoftodos,prependcaption,textsize=tiny]{todonotes} 

\theoremstyle{plain}
\newtheorem{thm}{Theorem}[section] 
\newtheorem{theorem}[thm]{Theorem}

\newtheorem{lemma}[thm]{Lemma}

\newtheorem{corollary}[thm]{Corollary}
\newtheorem{prop}[thm]{Proposition}

\newtheorem{proposition}[thm]{Proposition}

\newtheorem*{claim*}{Claim} 
\newtheorem*{thm*}{Theorem}

\theoremstyle{definition}

\newtheorem{notation}[thm]{Notation}

\theoremstyle{remark}

\newtheorem{remark}[thm]{Remark}

\newcommand{\codim}{\operatorname{codim}}

\newcommand{\rank}{\operatorname{rank}}

\newcommand{\Gl}{\operatorname{GL}}

\newcommand{\Mat}{\operatorname{Mat}}

\newcommand{\Span}{\operatorname{Span}}

\newcommand{\OG}{\operatorname{O}}
\newcommand{\Og}{\OG}
\newcommand{\PO}{\operatorname{PO}}
\newcommand{\PSp}{\operatorname{PSp}}

\newcommand{\m}{\mathfrak{m}}

\newcommand{\A}{\mathbb{A}}

\newcommand{\sh}[1]{\mathcal{#1}}

\newcommand{\GL}{\Gl}

\newcommand{\SO}{\operatorname{SO}}

\newcommand{\End}{\operatorname{End}}

\newcommand{\isom}{\cong}
\newcommand{\iso}{\isom}

\newcommand{\sm}{\setminus}

\newcommand{\Z}{\mathbb{Z}}

\newcommand{\ZZ}{\Z}

\renewcommand{\1}{{\rm 1\hspace*{-0.4ex}%
\rule{0.1ex}{1.52ex}\hspace*{0.2ex}}}

\newcommand{\Sp}{{\mathrm{Sp}}}

\newcommand{\CH}{{\operatorname{CH}}}

\newcommand{\Gr}{\mathrm{Gr}}

\newcommand{\<}[1]{{\langle}#1 {\rangle}}

\newcommand{\Tr}{\operatorname{Tr}}

\newcommand{\Iso}{\mathrm{Iso}}
\renewcommand{\m}[1]{\begin{bmatrix} #1 \end{bmatrix}}
\newcommand{\mat}[1]{\mathrm{Mat}_{#1\times#1}(k)}

\makeatletter
\def\bbordermatrix#1{\begingroup \m@th
  \@tempdima 4.75\p@
  \setbox\z@\vbox{%
    \def\cr{\crcr\noalign{\kern2\p@\global\let\cr\endline}}%
    \ialign{$##$\hfil\kern2\p@\kern\@tempdima&\thinspace\hfil$##$\hfil
      &&\quad\hfil$##$\hfil\crcr
      \omit\strut\hfil\crcr\noalign{\kern-\baselineskip}%
      #1\crcr\omit\strut\cr}}%
  \setbox\tw@\vbox{\unvcopy\z@\global\setbox\@ne\lastbox}%
  \setbox\tw@\hbox{\unhbox\@ne\unskip\global\setbox\@ne\lastbox}%
  \setbox\tw@\hbox{$\kern\wd\@ne\kern-\@tempdima\left[\kern-\wd\@ne
    \global\setbox\@ne\vbox{\box\@ne\kern2\p@}%
    \vcenter{\kern-\ht\@ne\unvbox\z@\kern-\baselineskip}\,\right]$}%
  \null\;\vbox{\kern\ht\@ne\box\tw@}\endgroup}
\makeatother

\bibliography{Article,BenW_Standard_BibTeX} 

\begin{document}

\title{Spaces of generators for matrix algebras with involution}
\author{Taeuk Nam}
\address{Taeuk Nam, Department of Mathematics, Harvard University, Cambridge~MA 02138, USA.}
\email{tnam@math.harvard.edu}
\author{Cindy Tan}
\address{Cindy Tan, Department of Mathematics, University of Chicago, Chicago, IL 60637, USA.}
\email{cindy@math.uchicago.edu}
\author{Ben Williams} 
\address{Ben Williams, Department of Mathematics, University of British Columbia, Vancouver~BC V6T 1Z2, Canada.}
\email{tbjw@math.ubc.ca}

\subjclass[2020]{15A63, 16W10}

\begin{abstract}
  Let $k$ be an algebraically closed field of characteristic different from $2$. Up to isomorphism, the algebra $\Mat_{n \times n}(k)$ can be endowed with a
  $k$-linear involution in one way if $n$ is odd and in two ways if $n$ is even. 
  
  In this paper, we consider $r$-tuples $A_\bullet \in \Mat_{n\times n}(k)^r$ such that the entries of $A_\bullet$ fail
  to generate $\Mat_{n\times n}(k)$ as an algebra with involution. We show that the locus of such $r$-tuples forms a 
  closed subvariety $Z(r;V)$ of $\Mat_{n\times n}(k)^r$ that is not irreducible. We describe the
  irreducible components and we calculate the dimension of the largest component of $Z(r;V)$ in all cases. This gives a
  numerical answer to the question of how generic it is for an $r$-tuple $(a_1, \dots, a_r)$ of elements in
  $\Mat_{n\times n}(k)$ to generate it as an algebra with involution.
\end{abstract}
\maketitle

\ifdraft{\tableofcontents}

\section{Introduction}

let $k$ be a field and let $n$ be a positive integer.
One may ask whether an $r$-tuple of matrices $(A_1, \dots, A_r)$ generates the matrix algebra $\Mat_{n \times n}(k)$
i.e., whether all matrices $B \in \Mat_{n \times n}(k)$ may be written as noncommutative polynomials in the $A_i$. This
question is a basic one in the study of matrix algebras. The landmark result in this area is a theorem of Burnside,
which says that over an algebraically closed field $A_1, \dots, A_r$ generate the algebra unless they share a nontrivial
invariant subspace: see e.g., \cite{Lam1998}, \cite{LomonosovsimplestproofBurnside2004}. From this it follows that a
generic pair of matrices $(A_1, A_2)$ generate the matrix algebra. Detecting when a given $r$-tuple generates the
algebra has been addressed in \cite{Laffey1981}, \cite{ShemeshCommoneigenvectorstwo1984}, \cite{AslaksenGeneratorsmatrixalgebras2009}, and
the problem of determining irredundant generating sets been considered by Laffey
\cite{LaffeyIrredundantgeneratingsets1983}. Recently, in \cite{first2020b}, it has proved useful to estimate ``how
generic'' a condition it is for an $r$-tuple to generate $\Mat_{n \times n}(k)$. That is, the set $Z$ of $r$-tuples that
fail to generate is a Zariski-closed subset of $\Mat_{n \times n}(k)^r$, and one wishes to know how large the
codimension of $Z$ in $\Mat_{n \times n}(k)^r$ is. The calculation has implications, as explained in \cite{first2020b},
for the number of generators required for forms of matrix algebras, the Azumaya algebras.

Let $k$ be an algebraically closed field of characteristic different from $2$.  In this paper, we consider the case of
$\Mat_{n \times n}(k)$ equipped with a $k$-linear involution $\rho$. The most important example of a $k$-linear
involution is the transpose operation, and if $n$ is odd then all involutions are isomorphic to this one. If $n$ is
even, then symplectic involutions also arise. An $r$-tuple $(A_1, \dots, A_r)$ generates the matrix algebra with
involution if all matrices $B \in \Mat_{n \times n}(k)$ may be written as a noncommutative polynomial in the $A_i$ and
the $\rho(A_i)$.

The structure $\Mat_{n \times n}(k)$ with transposition is of fundamental importance, and we start to describe the
geometry of the space of generating $r$-tuples of this structure. We are also motivated by applications in the vein of
\cite{first2020b}, bounding the number of generators required for Azumaya algebras with involution. Involutions on
Azumaya algebras, including central simple algebras, are of substantial interest, as attested by
\cite{SaltmanAzumayaalgebrasinvolution1978}, \cite{M.-AKnus1990}, \cite{Knusbookinvolutions1998}, \cite{Auel2019}. As is
evident from \cite{Knusbookinvolutions1998}, involutions of central simple algebras are not only studied for their own
sake, but also due to their close relationship with the study of linear algebraic groups. At the end of this
introduction, we outline an application of the present paper to the groups $\PO(n)$ and $\PSp(n)$.

\medskip

Our attention in this paper is devoted to the closed variety $Z(r;V)$ of $r$-tuples of matrices that fail to generate
$\Mat_{n \times n}(k)$ as an algebra with involution. The dimension of the largest irreducible component of $Z(r;V)$ is
a measure of ``how generic'' the condition of being generating is. It is also the key input into the machinery of
\cite[Section 9.2]{first2020b}, which provides estimates on how many elements suffice to generate a general Azumaya
algebra with involution.  We show in Corollaries \ref{cor:SymmDenseOpen} and \ref{cor:SkewSymmDenseOpen} that even a
generic singleton $A \in \Mat_{n \times n}(k)$ generates the matrix algebra with involution unless $n=2$ and the
involution is symplectic.

Burnside's theorem tells us that an $r$-tuple of matrices $(A_1, \dots, A_r)$ fails to generate
$\Mat_{n \times n}(k)$ as an algebra with involution if and only if $A_1, \dots, A_r, \rho(A_1), \dots , \rho(A_r)$ have
a common invariant subspace $W$ of dimension $d \in \{1, \dots, n-1\}$. We classify $r$-tuples that fail to
generate based on their invariant subspaces. The notation $Z(l,d,r;V)$ denotes the variety of $r$-tuples that have a
common invariant $d$-dimensional subspace where $\dim_k(W \cap W^\perp) = l$. When $l=d$,
the subspace is totally isotropic and  when $l=0$ the subspace is anisotropic. 
One of our main theorems,
Theorem \ref{th:closedCoverOfZsymm}, says that the varieties $Z(d,d,r;V)$ and $\overline{ Z(0,d,r; V)}$ as $d$ varies
account for all the irreducible components of $Z(r;V)$. Our other main theorem, Theorem \ref{th:mainDimCalc}, calculates
the dimensions of the varieties $Z(l,d,r;V)$ for all $l$ and $d$.

With these theorems in hand, we determine for all values of $r$ and $n$ what the largest-dimensional irreducible
components of $Z(r;V)$ are. The calculations are given in Proposition \ref{pr:dimSymmNot4} for the symmetric case and
Proposition \ref{pr:dimSkew-Symmetric} in the skew-symmetric case.

\medskip

Knowledge of the dimensions of the components of $Z(r; V)$ is of relevance to calculations about the automorphism
group $PG$ of $\Mat_{n \times n}(k)$ as an algebra with involution, as we now explain. There are two possibilities for $PG$: either the
involution is orthogonal, in
which case the automorphism group is $\PO(n;k)$, or it is symplectic, in which case it is $\PSp(n; k)$.

Since $\Mat_{n \times n}(k)$ is a generically free $PG$-representation, it can be used to produce an approximation to the classifying stack $BPG$
in the sense of \cite{TotaroChowringclassifying1999}. This approximation is obtained by first discarding a closed
subvariety including all the points where $PG$ does not act freely, in our case $Z(r;V)$, and then taking the quotient
by $PG$.

The codimension of $Z(r; V)$ in $\Mat_{n \times n}(k)$ influences how close the approximation
\[f_r : \frac{\Mat_{n \times n}(k) \sm Z(r;V)}{PG} \to BPG\] is. For instance, the induced map on Chow groups
$f_r^*: \CH^i_{PG} \to \CH^i(U(r; V)/PG)$ is an isomorphism when $i < rn^2- \dim_k Z(r; V)$, and we have calculated
$\dim_k Z(r; V)$ in this paper.


\subsection{Structure of the paper}
The paper begins with an expository section, section \ref{sec:Preliminary}, setting up notation and establishing a
decomposition result, Lemma \ref{decomposition}, that is used throughout the paper. Subsection \ref{ss:Grass} contains a
calculation of the dimension of Grassmannians of subspaces meeting certain isotropy conditions, which is not readily
available in the literature.

In section \ref{sec:GeneratingTuples}, we define $Z(r;V)$, the variety of nongenerating $r$-tuples in
$\End_k(V)$.  Burnside's theorem says that a set of endomorphisms
$\{A_1, \dots, A_r\}$ of $V$ generate the full endomorphism algebra if and only if there is no nonzero proper subspace
$W \subset V$ such that $A_i W \subseteq W$ for all $i$. A decomposition of $Z(r;V)$ into subsets
$Z(l,d,r; V)$ is given by using the dimension of common invariant subspaces $W$, denoted $d$, and the dimensions of
$W \cap W^\perp$, denoted $l$. Proposition \ref{pr:inclusionInZll} tells us that the union of all components
$Z(d,d,r;V)$ and $\overline{Z(0,d,r;V)}$ forms a closed cover of $Z(r;V)$. Subsection \ref{ss:rfewInv} is devoted to
showing that $r$-tuples exist in $Z(l,d,r;V)$ that are sufficiently generic, and then subsection
\ref{ss:irreducibleCompsOfZ(ldrV)} shows that in most cases, $Z(l,d,r;V)$ are irreducible varieties, the exception being
when the bilinear form on $V$ is symmetric and the invariant subspaces are isotropic of dimension one-half that of $V$.

In section \ref{sec:DimensionCalculation}, we calculate the dimensions of the varieties $Z(l,d,r; V)$, and therefore the
dimensions of the various components of $Z(r;V)$. Using the existence result of subsection \ref{ss:rfewInv} and the
dimension calculation, we rule out any unexpected containments among these closed
varieties. This allows us to establish Theorem \ref{th:closedCoverOfZsymm}, describing the irreducible components of
$Z(r;V)$ in every case. The rest of the paper, with a view to the needs of \cite{first2020b}, is devoted to the calculation of which among the irreducible components of
$Z(r;V)$ has the largest dimension. This calculation is done by elementary methods in Propositions \ref{pr:dimSymmNot4}
and \ref{pr:dimSkew-Symmetric}.

\subsection{Notational conventions} \label{sec:notConv}

Throughout:
\begin{itemize}[label={--}]
\item $k$ is an algebraically closed field of characteristic different from $2$.
\item $n$ and $r$ are positive integers.
\item $V$ denotes $k^n$ equipped with a nondegenerate bilinear form $\<{\cdot, \cdot}$ that is either symmetric or
skew-symmetric.
\item $V^*$ is the dual space, and if $f$ is a linear map, $f^*$ denotes the dual map.
\item $\phi$ is the induced isomorphism $\phi: V \to V^\ast$ given by $\phi(v)(w) = \<{v,w}$.
\item If $W \subseteq V$, then $W^\perp = \{ x \in V \mid \<{w,x} = 0,\, \forall w \in W \}$.
\item $G$ denotes the group of invertible endomorphisms of $V$ that preserve the bilinear form. This is isomorphic either to the (split) orthogonal group $O(n;k)$ or the symplectic group $\Sp(n;k)$
  over $k$, depending on whether $\<{\cdot, \cdot}$ is symmetric or skew-symmetric.
\item $PG$ denotes the adjoint group associated to $G$, i.e., the quotient of $G$ by its centre $Z(G)$.
\item $\bullet$ is used to indicate an $r$-tuple, so that $A_\bullet = (A_1, \dots, A_r)$.
\item $\delta_{i,j}$ denotes the \emph{Kronecker delta}, which is equal to $1$ if $i = j$ and equal to $0$ otherwise.
 \item $\Omega_{2m}$ denotes a matrix given by:
\[ \Omega_2 = \m{0&-1\\1&0}, \quad \text{and} \quad \Omega_{2m} = \m{
	\Omega_2 & & 0 \\
 & \ddots & \\
 0 & & \Omega_2}\]
\item A \textit{variety} means a reduced separated finite type $k$-scheme, not necessarily connected or
  irreducible. Varieties are endowed with the Zariski topology.
\item The \textit{dimension} of a variety is the maximum of the dimensions
  of all the irreducible components
\item The \textit{codimension} of $Z$ in $V$, which is only ever considered when $V$ is equidimensional, is
  the minimum of the codimensions of the irreducible components of $Z$.
\end{itemize}

\subsection*{Background on bilinear spaces and involutions} \label{ss:BSandInv}

Recall from
\cite[Introduction]{Knusbookinvolutions1998} that an \textit{involution} of $\Mat_{n \times n}(k)$ is an additive map
$\rho : \Mat_{n \times n}(k) \to \Mat_{n \times n}(k)$ such that
\[ \rho(AB) = \rho(B) \rho(A) \quad \forall A, B \in \Mat_{n \times n}(k). \]
We consider only $k$-linear involutions, which are called ``involutions of the first kind'' in
\cite{Knusbookinvolutions1998}.

\begin{notation} \label{not:adjointInv} Recall that $V$ denotes $k^n$ equipped with a particular bilinear form
  $\phi : V \to V^*$. Let $\rho$ denote the adjoint involution on $\End_k(V)$, given by
  $\rho(A) = \phi^{-1}\circ A^* \circ \phi$.
\end{notation}

Every $k$-linear involution of $\Mat_{n \times n}(k)$ is obtained from a nondegenerate symmetric or skew-symmetric bilinear form on
$k^n$ by this construction: see \cite[p.~1]{Knusbookinvolutions1998}. We write $\End_k(V)$
for the algebra-with-involution $\Mat_{n \times n}(k)$ endowed with the involution associated to $\<{\cdot, \cdot}$.

Two such algebras with involution $\End_k(V)$ are isomorphic if and only if the bilinear forms differ by a similitude,
\cite[Proposition 12.34]{Knusbookinvolutions1998}. The field $k$ is algebraically closed, however, so up to similitude,
there is a unique symmetric bilinear form on $k^n$, and there is a unique skew-symmetric bilinear form if and only if
$n$ is even. Therefore, up to isomorphism there are one or two $k$-linear involutions on $\Mat_{n \times n}(k)$,
depending on the parity of $n$.

If the bilinear form is symmetric, then the involution $\rho$ is said to be \textit{orthogonal}. This case includes the
involution $\rho(A) = A^\mathrm{T}$. If the bilinear form is skew-symmetric, then $\rho$ is \textit{symplectic}.

\section{Subspaces of Bilinear Spaces} \label{sec:Preliminary}

Everything in this section is well known, and we make no claim to originality. 
We include this section to provide
convenient reference for later sections.

\subsection{Preliminary results}

The following omnibus lemma collects a number of useful facts about $\perp$. All proofs are elementary.
\begin{lemma}
  For all subspaces $W_1, W_2$ of $V$, the following identities hold.
  \begin{enumerate}
  \item $\dim W_1 + \dim W_1^\perp = n$;
  \item $(W_1 + W_2)^\perp = W_1^\perp \cap W_2^\perp$;
  \item $(W_1 \cap W_2)^\perp = W_1^\perp + W_2^\perp$;
  \item $(W_1^\perp)^\perp = W_1$;
  \end{enumerate}
  as well as the relation
  \begin{enumerate}
    \setcounter{enumi}{4}
  \item $W_1 \subseteq W_2$ if and only if $W_1^\perp \supseteq W_2^\perp$.
  \end{enumerate}
\end{lemma}

\begin{notation} \label{not:IsoRank}
 Let $W\subseteq V$ be a subspace. Write $\Iso(W) = W \cap W^\perp$. We call this the \textit{isotropic radical} of
 $W$. We call $\dim_k \Iso(W)$ the \textit{isotropy rank} of $W$.
 
 We say that a subspace $W \subseteq V$ is \textit{nondegenerate} if the restriction of $\<{\cdot, \cdot}$ to $W$ is
 nondegenerate, i.e., if $\Iso(W) = \{0\}$.
\end{notation}



\begin{lemma} \label{lem:IsoComplement}
 Suppose $W \subseteq V$ is a totally isotropic subspace, i.e., $W = \Iso(W)$. Then there exists a totally isotropic
 subspace $C$ such that $\<{ \cdot, \cdot} $ induces a perfect pairing $W \times C \to k$.
\end{lemma}

 Observe that the conclusion of the lemma implies that $W \cap C = \{0\}$. 
\begin{proof} 
 Choose a basis $\{ b_1, \dots, b_d \}$ of $W$. By induction, we produce a basis $\{c_1, \dots, c_d\}$ for $C$ such that $\<{b_i,
 c_j} = \delta_{ij}$ and such that $\<{c_i , c_j} = 0$ for all $i, j$. Suppose
 $\{ c_1, \dots, c_i\}$ has already been produced for some $i\ge 0$.

 The elements $\{ b_1, \dots, b_d, c_1, \dots, c_i\}$ form a linearly independent set, therefore so too do the dual
 elements $\{ \phi(b_1), \dots, \phi(b_d), \phi(c_1), \dots, \phi(c_i)\}$. We can find some
 $c_{i+1}' \in V$ such that
 \begin{align}
 \label{eq:5}
 \<{b_j, c_{i+1}'} &= \delta_{j, i+1}\text{ for all }j \in \{1, \dots, d\} \\
 \label{eq:5b} \<{c_j, c_{i+1}'} &= 0\text{ for all }j \in \{1, \dots, i\}
 \end{align}
 Set $\<{ c_{i+1}' , c_{i+1}' } = 2\lambda$. Let $c_{i+1} = c_{i+1}' - \lambda b_{i+1}$. Observe that $c_{i+1}$ also
 satisfies equations \eqref{eq:5} and \eqref{eq:5b} and $\<{c_{i+1}, c_{i+1}} = 0$.

 We define $C$ to be the subspace generated by $\{c_1, \dots, c_d\}$. This space has the required properties, as is
 easily verified.
\end{proof}

\subsection{Decomposition of bilinear spaces}

Now we show that any subspace $W \subseteq V$ induces a decomposition of $V$. We will use this in order to
stratify the Grassmannians $\Gr(d,V)$ of $d$-planes in $V$ according to the isotropy rank of the subspaces. The main
result here is Proposition \ref{lem:nice-basis}, which allows us to write any inclusion of a subspace $W \subseteq V$ in
a standard form. 

\begin{notation}
 Throughout the rest of this section, fix a subspace $W$ of $V$ and let $W_a$ and $W_b$ be subspaces of $V$ such that:
 \begin{align*}
 \Iso(W) \oplus W_a &= W, \\
 \Iso(W) \oplus W_b &= W^{\perp}.
 \end{align*}
\end{notation}

\begin{lemma}\label{decomposition}
  There exists a subspace $C$ of $V$ such that:
  \begin{itemize}
  \item  $\Iso(W) \oplus W_a \oplus W_b \oplus C = V$;
  \item The four subspaces $\Iso(W)$, $W_a$, $W_b$ and $C$ are pairwise orthogonal with the exception of the pair
    $\Iso(W)$ and $C$;
  \item $C$ is totally isotropic and $\<{\cdot, \cdot}$ induces a perfect pairing $\Iso(W) \times C \to k$.
  \end{itemize}
\end{lemma}
\begin{proof}
 Observe that $W_a \oplus W_b$ is nondegenerate. As a consequence $(W_a \oplus W_b)^\perp$ is nondegenerate as
 well. Since $\Iso(W) \subseteq (W_a \oplus W_b)^{\perp}$, we can apply Lemma \ref{lem:IsoComplement} to produce
 $C \subseteq (W_a \oplus W_b)^{\perp}$ such that $C$ is totally isotropic and $\<{\cdot, \cdot}$ induces a perfect
 pairing $\Iso(W) \times C \to k$. We remark that $C$ has trivial intersection with $\Iso(W)$, $W_a$ and $W_b$.

 To complete the argument, we count dimensions: The dimension of $W$ is $d$ and that of $\Iso(W)$ is $l$. The dimension
 of $W_a$ is $d-l$ and the dimension of $W^\perp$ is $n-d$. Therefore the dimension of $W_b$ is $n-d-l$. Adding
 together, we see that the dimension of
 \[ \Iso(W) \oplus W_a \oplus W_b \oplus C \]
 is $l + (d-l) + (n-d-l) + l = n$. The result follows.
\end{proof}

For a basis $\mathcal{B} = \{b_i\}_{i=1}^n$ of $V$, let $Q(\mathcal{B})$ be the matrix
$Q(\mathcal{B})_{ij} = \<{b_i,b_j}$. We may abbreviate this to $Q$ if the basis is understood.

\begin{prop}\label{lem:nice-basis}
	 Let $W \subseteq V$ be a subspace with dimension $d$ and isotropy rank $l$. Then there exists an ordered basis $\mathcal{B}$ of $V$ so that $W$ is the span of the first $d$ basis elements and 
	\begin{gather*}
	Q(\mathcal{B}) = \m{0 & 0 & e I_l \\ 0 & M & 0 \\ I_l & 0 & 0},
	\end{gather*}
	where $M = I_{n-2l}$ and $e = 1$ if $\<{\cdot, \cdot}$ is symmetric, and $M = \Omega_{n-2l}$ and $e = -1$ if
 $\<{\cdot, \cdot}$ is skew-symmetric.
\end{prop}

\begin{proof}
	By Lemma \ref{decomposition}, we may decompose $V = \Iso(W) \oplus W_a \oplus W_b \oplus C$ so that $W_a$ and
 $W_b$ are anisotropic, $\Iso(W)$, $C$ are totally isotropic, and $\Iso(W)$, $W_a$, $W_b$, and $C$ are all pairwise
 orthogonal with the exception of $\Iso(W)$ and $C$, where $\<{\cdot, \cdot}$ restricts to a perfect pairing
 $\Iso(W) \times C \to k$.

 The construction of the basis $\mathcal B$ is routine from this point.
	%
	%
	%
\end{proof}

\begin{corollary}\label{lem:nice-basis-alt}
  Suppose $\<{\cdot, \cdot}$ is skew-symmetric. Write the dimension of $V$ as $n = 2m$. Let $W \subseteq V$ be a
  subspace with dimension $d$ and isotropy rank $n-2t$ for some $t$. Then there exists an ordered basis
  $\mathcal{B} = \{b_i\}_{i=1}^n$ of $V$ so that $W = \mathrm{span}(\{b_i\}_{i=1}^{l+t} \cup \{b_i\}_{i=m+1}^{m+t})$ and
	\begin{gather*}
	Q(\mathcal{B}) = \m{0 & -I_m \\ I_m & 0}.
	\end{gather*}
\end{corollary}

\begin{proof}
 The basis $\mathcal{B}$ may be constructed by taking the ordered basis constructed by Lemma \ref{lem:nice-basis} and
 rearranging the elements $x_1,\dots,x_{n-2l}$.
\end{proof}

\subsection{Grassmannians} \label{ss:Grass}

In this section, we collect some well-known facts about the Grassmannian of $d$-dimensional subspaces of $V$ of
prescribed isotropy rank. Most elementary references for Grassmannians of $d$-planes in bilinear spaces $V$ concentrate
on the case of maximal isotropy: the Isotropic Grassmannians.

Consider the Grassmannian $\Gr(d; V)$ of $d$-planes in $V$. We abuse notation and write $V$ for the trivial bundle with
value $V$ on $\Gr(d; V)$. This Grassmannian carries a canonical sub-bundle $\sh S \hookrightarrow V$ of rank $d$, and
via the inner-product isomorphism, there is a map
  \[\psi: \sh S \to V \to V^* \to \sh S^*.\]

  Define $\Gr(\ge l, d; V)$ to be the closed subvariety of $\Gr(d; V)$ where $\rank \psi \le d-l$. That is,
  $\Gr(\ge l, d; V)$ is the variety of $d$-dimensional subspaces of $V$ having isotropy rank $\ge l$. Then define
  \[ \Gr(l, d; V) = \Gr(\ge l, d; V) \sm \Gr(\ge l+1, d; V).\] This is a locally closed subvariety of $\Gr(d;V)$,
  consisting of subspaces of $V$ of dimension $d$ and isotropy rank $l$.

\begin{lemma}\label{lem:transitive}
 Let $U,W$ be subspaces of $V$ with the same dimension and isotropy rank. Then there exists $g \in G$ such that $gU = W$
 and $g U^\perp = W^\perp$.
\end{lemma}

\begin{proof}
 By Lemma \ref{lem:nice-basis}, there exist bases $\mathcal{B} = \{b_i\}_{i=1}^n$ and $\mathcal{C} = \{c_i\}_{i=1}^n$
 of $V$ such that $U = \Span(\{b_i\}_{i=1}^d\})$, $W = \Span(\{c_i\}_{i=1}^d\})$, and similarly for
 $U^\perp$ and $W^\perp$, and such that $Q(\mathcal{B})=Q(\mathcal{C})$. Let $g \in \GL_n(k)$ be such that
 $gb_i = c_i$ for all $i$. Since $Q(\mathcal{B}) = Q(\mathcal{C})$, it follows that
 $\<{b_i,b_j} = \<{c_i,c_j} = \<{gb_i,gb_j}$. Therefore, $g$ preserves $\<{\cdot,\cdot}$, which is
 to say $g$ is in $G$.
\end{proof}

\begin{proposition}[Dimension of the symmetric Grassmannians]\label{prop:gr-dim-sym}
Suppose $\<{\cdot, \cdot}$ is symmetric. Then
\[\dim_k \Gr(l,d;V) = d(n-d) - \frac{l^2 + l}{2}.\]
\end{proposition}

\begin{proof}
 The group $G$ acts on $\Gr(l,d;V)$ transitively, by Lemma \ref{lem:transitive}.
	
 Let $W \subseteq V$ be a subspace of dimension $d$ and isotropy rank $l$. Let $H \subseteq G$ be the stabilizer of $W$
 under the action of $G$. By the orbit-stabilizer theorem,
 \[\dim_k \Gr(l,d;V) = \codim_k(H \hookrightarrow G).\]
	
 By Lemma \ref{lem:nice-basis}, there exists a basis $\mathcal{B} = \{b_i\}_{i=1}^n$ of $V$ so that
 $W = \mathrm{span}(\{b_i\}_{i=1}^d)$ and
	\begin{gather*}
 Q(\mathcal{B}) = [\<{b_i,b_j}_\phi] = \m{0 & 0 & I_l \\ 0 & I_{n-2l} & 0 \\ I_l & 0 & 0}.
	\end{gather*}
	
	We now write automorphisms of $V$ as matrices with respect to the basis $\mathcal{B}$.
	
	Observe that $A \in G$ if and only if $A^TQA = Q$. The Lie algebra $\mathfrak{g}$ of $G$ is the set of matrices
 $X$ satisfying $X^TQ + QX = 0$.
	
	Let $X \in \mat{n}$. Write $X$ in block form:
	\begin{gather*}
		X = \m{X_{11} & X_{12} & X_{13} \\ 
		X_{21} & X_{22} & X_{23} \\
		X_{31} & X_{32} & X_{33}}
	\end{gather*}
	where $X_{11}$ and $X_{33}$ are $l \times l$ matrices and $X_{22}$ is an $(n-2l)\times(n-2l)$ matrix. Then the
 condition $X^TQ + QX = 0$ is equivalent to the conditions:
	\[X_{13} = -X_{13}^T,\ X_{22} = -X_{22}^T,\ X_{31} = -X_{31}^T,\ X_{11} = -X_{33}^T,\ X_{12} = -X_{23}^T,\ X_{21} = -X_{32}^T\]
	This implies that $X_{13}$, $X_{22}$, and $X_{31}$ are skew-symmetric, and that the entries of $X_{23}$,
 $X_{32}$, and $X_{33}$ are entirely determined by the entries in $X_{12}$, $X_{11}$, and $X_{21}$
 respectively. Therefore, there are $\frac12(n^2-n)$ free entries in a matrix $X \in \mathfrak{g}$.
	
	The group $H$ is the subgroup of $G$ consisting of matrices $A \in G$ satisfying the further condition that $AW
 = W$.
 Then the Lie algebra $\mathfrak{h}$ of $H$ consists of all matrices $X \in \mathfrak{g}$ such that
 $XW \subseteq W$.
	
	If $X \in \mathfrak{h}$, the additional condition of $XW \subseteq W$ fixes some number of the $\frac12(n^2-n)$
 free entries. The number of entries fixed is the codimension of $\mathfrak{h}$ in $\mathfrak{g}$, which is the
 dimension of the Grassmannian of interest.
	
	We may choose the free entries to be the blocks $X_{11}$, $X_{12}$, $X_{21}$, and the entries beneath the
 diagonal in $X_{13}$, $X_{22}$, and $X_{31}$. Note that $XW \subseteq W$ if and only if $X$ has the block
 form $$X = \m{* & * \\ 0 & *},$$ where the lower-left block is a $d \times (n-d)$ block of zeros.
	
	The block $X_{31}$ is entirely 0, so $\frac12(l^2-l)$ free entries are fixed in this block.
	
	The lower-left $(n-d-l) \times (d-l)$ block of $X_{22}$ is entirely 0. The zeros are all below the diagonal, so
 this fixes $(n-d-l)(d-l)$ entries.
	
	The lowest $n-d-l$ rows of $X_{21}$ and the leftmost $d-l$ rows of $X_{32}$ are entirely 0. By the relation
 $X_{21} = -X_{32}^T$, this means that $X_{21}$ and $X_{32}$ are both entirely 0. $X_{21}$ is free, so $(n-2l)l$
 entries are fixed here.
	
	In total, there are $$\frac12(l^2-l) + (n-d-l)(d-l) + (n-2l)l = d(n-d) - \frac{l^2+l}{2}$$ free entries fixed by
 the additional condition $XW \subseteq W$. The result follows.
\end{proof}

\begin{proposition}[Dimension of the skew-symmetric Grassmannians] \label{pr:dimAltGr} Suppose $\<{\cdot, \cdot}$ is skew-symmetric. Then
 \[ \dim_k \Gr(l,d;V) = d(n-d) - \frac{l^2 - l}{2}.\]
\end{proposition}

\begin{proof}
 The proof is analogous to that of Proposition \ref{prop:gr-dim-sym}.
	
 Again, $G$ is the group of automorphisms of $V$ that preserve the bilinear form $\<{\cdot,\cdot}$. The group $G$ acts
 on $\Gr(d-l,d;V,\phi)$. The action is transitive by Lemma \ref{lem:transitive}. Fix a subspace $W$ of dimension $d$,
 and isotropy rank $l$. Let $t = (d-l)/2$. Let $H \subseteq G$ be the stabilizer of $W$ under the action of $G$.
	
 Write $m=n/2$. By Lemma \ref{lem:nice-basis-alt}, there is a basis $\mathcal{B} = \{b_i\}_{i=1}^n$ of $V$ so that
 $W = \mathrm{span}(\{b_i\}_{i=1}^{l+t} \cup \{b_i\}_{i=m+1}^{m+t})$ and
 \begin{gather*}
 Q(\mathcal{B}) = \m{0 & -I_m \\ I_m & 0}.
	\end{gather*}
	
	The Lie algebra $\mathfrak{g}$ consists of all matrices $X$ for which $X^TQ + QX = 0$. Let $X \in \mat{n}$ and
 write
	\begin{gather*}
		X = \m{X_{11} & X_{12} \\ X_{21} & X_{22}}.
	\end{gather*}
	where each $X_{ij}$ is an $m \times m$ matrix. The condition $X^TQ+QX = 0$ is equivalent to the conditions that $X_{11} = -X_{22}^T$, and $X_{12}$ and $X_{21}$ are symmetric. Therefore, $X$ has $\frac12(n^2-n)$ free entries. Let us say the free entries are $X_{11}$ and the entries on and below the diagonals in $X_{12}$ and $X_{21}$. 
	
	The Lie algebra $\mathfrak{h}$ consists of matrices in $\mathfrak{g}$ satisfying $XW \subseteq W$. The number of free entries fixed by this additional condition is the codimension of $\mathfrak{h} \hookrightarrow \mathfrak{g}$ and also the dimension of the Grassmannian.
	
	Let $N = \{i: b_i \in W\}$. Then $W = \mathrm{span}(\{b_i: i \in N\})$. Write the entries of $X$ as $x_{ij}$. The condition that $XW \subseteq W$ is equivalent to the condition that $x_{ij} = 0$ if $i \in N$ and $j \notin N$.
	
	In $X_{11}$, the bottom left $(m-l-t)\times(l+t)$ entries are 0. In addition, notice that the bottom left
 $(m-t)\times t$ entries of $X_{22}$ are 0. By the relation $X_{11} = -X_{22}^T$, the upper right $t\times(m-t)$
 entries of $X_{11}$ are 0. The two blocks do not overlap so there are $(m-l-t)(l+t) + t(m-t)$ fixed entries in $X_{11}$.
	
	In $X_{12}$, the bottom $(m-l-t)\times t$ entries are 0. All entries in this block are either on or below the diagonal of $X_{12}$, so this entire block was free to begin with. Thus there are $t(m-l-t)$ entries fixed in $X_{12}$.
	
	In $X_{21}$, the bottom $(m-t)\times(l+t)$ entries are 0. There are $\frac12(l^2-l)$ entries in this block strictly above the diagonal of $X_{21}$. Thus, there are $(m-t)(l+t)-\frac12(l^2-l)$ initially free entries that are fixed in $X_{21}$.
	
	Altogether, we have $$(m-l-t)(l+t) + t(m-t) + t(m-l-t) + (m-t)(l+t)-\frac12(l^2-l) = d(n-d) - \frac{l^2-l}{2}$$ entries fixed by the additional condition that $XW \subseteq W$. The result follows.
\end{proof}

\section{Generating tuples} \label{sec:GeneratingTuples}

Let $U(r;V) \subseteq \End_k(V)^r$ be the set of $r$-tuples of elements that generate $\End_k(V)$ as an algebra with
involution. We define $Z(r;V) = \End_k(V)^r \setminus U(r;V)$. We may identify $\End_k(V)$ with $\A_k^{n^2r}$, so
$\End_k(V)$ is an affine variety. We will see in Proposition \ref{pr:Zclosed} that $Z(r;V)$ is closed in $\End_k(V)$, and
so $U(r;V)$ is open.

The involution $\rho$ is defined by $\rho(A) = \phi^{-1} \circ A^* \circ \phi$. It satisfies
\[ \<{ v, A w} = \phi(v) Aw = A^*(\phi(v)) w = \<{\rho(A), v}, \qquad \forall A \in \End_k(V), \, v,w \in V. \] We will
say that $W$ is a \textit{$\rho$-invariant} subspace of $A$ if $A W \subseteq W$ and $\rho(A) W \subseteq W$. If
$A_\bullet$ is an $r$-tuple of endomorphisms, then $W$ will be said to be a $\rho$-invariant subspace of $A_\bullet$ if
it is $\rho$-invariant for each $A_i$. If $A_\bullet \in \End(V)$, then write $\rho(A_\bullet)$ for the $r$-tuple
obtained by applying $\rho$ termwise to $A_\bullet$.

We are interested in the dimension of the irreducible components of $Z(r;V)$. We recall a theorem that was mentioned in
the introduction.

\begin{thm}[Burnside]
\label{burnside}
Let $A_\bullet =(A_1, ..., A_r) \in \End_k(V)^r$. Then $A_\bullet$ generates $\End_k(V)$ as an algebra if and only if
there does not exist any nonzero proper subspace $W\subsetneq V$ such that $A_i(W) \subset W$ for all
$1 \leq i \leq r$.
\end{thm}

\begin{corollary}
\label{involution}
Let $A_\bullet \in \End_k(V)^r$. Then $A_\bullet$ generates $\End_k(V)$ as an algebra with involution if and only if
there does not exist any nonzero proper subspace $W\subsetneq V$ that is $\rho$-invariant for $A_\bullet$.
\end{corollary}

\begin{lemma} \label{lem:rhoperp} 
 If $A \in \End_k(V)$ and $W \subset V$ then $\rho(A)W \subseteq W$ if and only if $A W^\perp \subseteq W^\perp$.
\end{lemma}
\begin{proof}
  This follows from the identity
  \[ \<{\rho(A)v, w} = \<{v, Aw } \]
  and the equality $(W^\perp)^\perp = W$.
%
\end{proof}



  Let $W \subset V$ be a subspace. We define $Z(W, r; V) \subset \End_k(V)^r$ to be the set of $r$-tuples $A_\bullet$
  having $W$ as a $\rho$-invariant subspace. By Lemma \ref{lem:rhoperp}, these are the $r$-tuples such that $A_i(W) \subseteq W$ and
  $A_i(W^\perp) \subseteq W^\perp$ for all $i \in \{1, \dots, r\}$.

 Since $(W^\perp)^\perp = W$, there is an equality $Z(W, r;V) = Z(W^\perp, r;V)$. It is therefore sufficient to
 consider only the cases where $\dim W \le n/2$.

\begin{notation}
Let $1 \leq d \leq n/2$. We define
\begin{equation}
 Z(d, r; V) = \bigcup_{\dim(W) = d} Z(W, r; V).
\end{equation}
where the union is taken over all dimension-$d$ subspaces $W \subset V$.
By Corollary \ref{involution},
\begin{equation}
 Z(r; V) = \bigcup_{d = 1}^{n-1} Z(d, r; V).
\end{equation}
We decompose each $Z(d, r; V)$ further, into a finite number of sets of $r$-tuples defined by the isotropy of the $\rho$-invariant subspace.

\medskip 

Let $0 \leq l \leq d$. We define
\begin{equation}
 Z( \ge l, d, r; V) = \bigcup_{\substack{\dim(W) = d \\ \dim\Iso(W) \ge  l}} Z(W, r; V).
\end{equation}
where the union is taken over all subspaces $W \subset V$ having dimension $d$ and isotropy rank $l$ or more. Note that
$Z(d,r; V) = Z(\ge 0, d,r; V)$.

 Finally, define $Z(l, d, r; V)$ to be $Z(\ge l, d, r; V) \sm Z(\ge l+1 , d ,r, V)$. This is the set of $r$-tuples of
 endomorphisms having a common $\rho$-invariant $d$-dimensional subspace of isotropy rank $l$, and no common $\rho$-invariant
 $d$-dimensional subspace of greater isotropy rank.
\end{notation}

\begin{table}[h]
  \centering
  {\small \begin{tabular}{|c||l|l|}
    \hline
    Notation & Description \\ \hline\hline
    $Z(r;V)$ & All nongenerating $r$-tuples  \\ \hline
    $Z(W, r;V)$ & $r$-tuples for which $W$ is $\rho$-invariant   \\ \hline
    $Z(d, r ; V)$ & $r$-tuples having some $d$-dimensional $\rho$-invariant subspace
    \\ \hline
    $Z(\ge l ,d, r; V)$ & $r$-tuples having some $d$-dimensional $\rho$-invariant subspace of isotropy rank at least $l$ \\ \hline
    $Z(l, d, r ;V)$ & $Z(\ge l , d, r ; V) \sm Z(\ge l+1, d,r; )$ \\
    \hline
    \end{tabular} }
                                                            \caption{Table of notations for nongenerating loci.}
  \label{tab:Zs}
\end{table}

We have decomposed $Z(r;V)$ into a finite number of sets $Z(l, d, r; V)$, which are empty unless $l \le d$ and $l \le n/2 -d$.

\begin{proposition} \label{pr:Zclosed}
 The varieties $Z(\ge l,d,r; V)$ are closed subvarieties of $\End_k(V)^r \iso \A^{n^2 r}_k$. Consequently, the sets
 $Z(l,d,r; V)$ are locally closed quasi-affine subvarieties of $\End_k(V)^r$.
\end{proposition}
\begin{proof}
 Let $X \subset \End_k(V)^r \times \Gr(\ge l, d, V)$ be the incidence variety consisting of pairs $(A_\bullet, W)$
 where $W$ is a $d$-dimensional subspace of $V$ that has isotropy rank no less than $l$ and is $\rho$-invariant under $A_\bullet$.

 We claim that $X$ is a closed subvariety of $\End_k(V)^r \times \Gr(\ge l, d, V)$. Since $\Gr(\ge l, d, V)$ is a closed
 subvariety of $\Gr(d, V)$, it carries two universal locally free sheaves, the universal subsheaf
 $\sh P \subseteq V$ of rank $d$ and the universal quotient $\sh Q = V/ \sh P$. We denote their pull-backs to
 $\End_k(V)^r \times \Gr(\ge l, d, V)$ by the same letters. On the other hand, $\End_k(V)^r$ carries
 $2r$ endomorphisms of $V$, corresponding to the $A_i$ and their involutions $\rho(A_i)$. We can produce a map of sheaves
 \[\xymatrix{ \sh P^{2r} \ar[r] & V^{2r} \ar^{(A_1, \dots, A_r, \rho(A_1), \dots, \rho(A_r))}[rrrr]& & & & V^{2r}
     \ar[r] & \sh Q^{2r} }\]
 on the product space, $\End_k(V)^r \times \Gr(\le s, d, V)$. The locus where this map of locally free sheaves vanishes
 is precisely the incidence variety $X$. It is therefore a closed subvariety of $\End_k(V)^r \times \Gr(\ge l, d, V)$.

 Since $\Gr(\ge l, d, V)$ is a closed subvariety of $\Gr(d, V)$, it is proper, and so the image of $X$ under the
 projection to $\End_k(V)^r$ is closed. This image is $Z(\ge l, d, r; V)$.
\end{proof}

The following two propositions demonstrate containment relations among the $Z(l,d,r;V)$. It will follow from Proposition
\ref{pr:tupleExists} that these are all the containment relations.

\begin{proposition} \label{pr:dualityEquality}
  Let $d \le n$ and $r \ge 1$. For all $l \ge 0$, the bijective correspondence $W \leftrightarrow W^\perp$ implies
  \[ Z(l, d, r; V) = Z(l, n-d, r ;V). \]
\end{proposition}

The proof is immediate.

\begin{proposition} \label{pr:inclusionInZll}
  Let $d \le n$ and $r \ge 1$. For all $l > 0$, there is an inclusion
  \[ Z(l,d,r; V) \subseteq Z(l,l, r ; V). \]
\end{proposition}
\begin{proof}
  It is worth noting that $Z(l,l,r; V) = Z(\ge l , l , r ;V)$ is a closed subvariety of $\End_k(V)$ consisting of
  $r$-tuples $A_\bullet$ such that there exists an $l$-dimensional totally isotropic subspace $\rho$-invariant under $A_\bullet$.
  
  If an $r$-tuple $A_\bullet$ is such that $W$ of dimension $d$ and isotropic rank $l$ is $\rho$-invariant under
  $A_\bullet$, then $\Iso(W)$, a totally isotropic subspace of dimension $l$ is also $\rho$-invariant under $A_\bullet$.
\end{proof}

\subsection{Tuples having a given invariant subspace}

\begin{proposition} \label{pr:dimOfZWr}
 Let $W \subseteq V$ be a subspace of rank $d$ and isotropy rank $l$. Then $Z(W,r; V)$ is isomorphic to an affine space
 of dimension
 \[\dim Z(W, r ; V) = r( (n-d)^2 + d^2 + l^2). \]
\end{proposition}
\begin{proof}
 The space $Z(W, r; V)$ is isomorphic to the $r$-fold product $Z(W, 1; V)^r$. Therefore it suffices to show that $Z(W,
 1; V)$ is isomorphic to an affine space of dimension $(n-d)^2 + d^2 + l^2$. We may decompose $V$ as in Lemma
 \ref{decomposition} into four subspaces: $\Iso(W), W_a, W_b, C$ of dimensions $l$, $d-l$, $n-d-l$ and $l$
 respectively.
 
 An endomorphism $A \in \End_k(V)$ lies in $Z(W, 1; V)$ if and only if it satisfies the following conditions:
 \begin{enumerate}
 \item $A(W_a) \subseteq \Iso(W) \oplus W_a = W$
 \item $A(W_b) \subseteq \Iso(W) \oplus W_b = W^\perp$
 \item $A(\Iso(W)) \subset \Iso(W)$, since $\Iso(W) = W \cap W^\perp$.
 \end{enumerate}
 If we choose a basis 
 \[ \sh B = \{\overbrace{b_1, \dots , b_l}^{\Iso(W)}, \overbrace{b_{l+1}, \dots, b_d}^{W_a}, \overbrace{b_{d+1},
 \dots , b_{n-l}}^{W_b}, \overbrace{b_{n-l}, \dots, b_n}^C \} \]
 of $V$ where the indicated elements are drawn from the indicated subspaces, then an endomorphism $A$ lies in $Z(W,1)$
 if and only if it takes the following matrix form:
 \begin{equation}
 \label{eq:4}
\bbordermatrix{
 ~ & \Iso(W) & W_a & W_b & C \cr \Iso(W) & A_{11} & A_{12} & A_{13} & A_{14} \cr W_a & 0 & A_{2,2}
 & 0 & A_{2,4} \cr W_b & 0 &0 & A_{3,2} & A_{3,4} \cr C & 0 & 0 & 0 & A_{4,4} }.
 \end{equation}
 Therefore the set $Z(W,1)$ is isomorphic to affine space of dimension
 \[ l^2 + (d-l)d + (n-d-l)(n-d) + nl = (n-d)^2 + d^2 + l^2 \] as asserted.
\end{proof}

\subsection{\texorpdfstring{$r$}{r}-tuples having few invariant subspaces} \label{ss:rfewInv}

This subsection is technical, but the main result, Proposition \ref{pr:tupleExists}, should not be seen as
surprising. Suppose $W \subseteq V$ is a subspace. Consider the variety of endomorphisms $A$ of $V$ such that $W$ is
$\rho$-invariant. Any such $A$ must also have $0$, $\Iso(W)$, $W^\perp$, $W+W^\perp$ and $V$ as invariant subspaces. We
show that an $A$ exists with precisely these invariant subspaces and no others, and as an added benefit, $A$ can be
chosen to have distinct eigenvalues. From there, it is immediate that an $r$-tuple $A_\bullet$ exists for which $\{0,
\Iso(W), W, W^\perp, W+W^\perp, V\}$ is precisely the set of $\rho$-invariant subspaces.

\begin{lemma} \label{lem:nonderogatory} Let $A: V \to V$ be a linear transformation with $n$ distinct eigenvalues. Then
  $W$ is an invariant subspace for $A$ if and only if $W$ is a direct sum of eigenspaces of $A$. In particular, there
  are only finitely many invariant subspaces.
\end{lemma}

\begin{proof}
  If $W$ is a direct sum of eigenspaces, then it is invariant.
  
  Conversely, suppose $W$ is invariant. The minimal polynomial of $A|_W$ is a factor of that of $A$. Consequently, $W$
  decomposes as a direct sum of eigenspaces of $A|_W$, which are also eigenspaces of $A$.
\end{proof}

  Suppose $\sh B \subset V$ is a subset, and $U \subseteq V$ is a subspace. We will say $U$ is \textit{subordinate} to
  $\sh B$ if $U$ is the span of a subset of $\sh B$.

Fix a subspace $W \subseteq V$. Choose an ordered basis $\sh B_{0}$ for $\Iso(W)$

\begin{lemma} \label{lem:longLemma}
  There exists an extension of $\sh B_{0}$ to an ordered basis $\sh B = \sh B_0 \cup \sh B_1 \cup \sh B_2 \cup \sh
  B_3$ for $V$ with the property that
  \begin{enumerate}
  \item $\sh B_0 \cup \sh B_1$ is a basis for $W$,
  \item $\sh B_0 \cup \sh B_2$ is a basis for $W^\perp$,
    \item \label{thirdItem} If $U$ and $U^\perp$ are both subordinate to $\sh B$, then $U \in \{0, \Iso(W), W, W^\perp, W+ W^\perp , V \}$.
  \end{enumerate}
\end{lemma}
\begin{proof}
  Suppose $\sh A$ is a linearly independent ordered subset of $V$ that contains $\sh B_{0}$ as an initial subset. Say
  $\sh A$ is \textit{good} if, for all $U$ subordinate to $\sh A$, either
  $U \in \{0, \Iso(W), W, W^\perp, W+W^\perp, V\}$ or $U^\perp$ is not subordinate to $\sh A$. Observe that condition
  \ref{thirdItem} is precisely the claim that the basis $\sh B$ is good.
  
  We will say $\sh A$ is \textit{very good} if it is good, and for any $U$ subordinate to $\sh A$ and containing
  $\Iso(W)$, either $U \in \{\Iso(W), W, W^\perp, W+W^\perp, V\}$ or $U^\perp \cap W$ is not subordinate to $\sh A$.

  We note that $\sh B_{0}$ itself is very good, since any subordinate subspace $U$, by virtue of being a subspace of
  $\Iso(W)$, satisfies $U^\perp \supseteq W + W^\perp \supseteq \Iso(W)$, and therefore $U^\perp$ or $U^\perp \cap W$
  can only be subordinate to $\sh B_0$ if it they are $\Iso(W)$.
  
  Now suppose $\sh A \subseteq W$ is very good and $\sh B_0 \subseteq \sh A$, but $\sh A$ is not a basis for $W$. We
  claim we may extend $\sh A$ to $\sh A \cup \{w\}$ in $W$, where the latter is also very good. The following conditions
  on an element $w \in W$ determine nonempty Zariski-open subsets of $W$:
  \begin{enumerate}
  \item $w \not \in \Span \sh A$,
  \item $\<{w, w} \neq  0$---this condition is open, and if it determines the empty set, then $W= W^\perp =
    \Iso(W)$, so $\sh B_{0}$ already spans $W$,
  \item \label{p3} for all $U$ subordinate to $\sh A$ and for which $U^\perp \not \supseteq W$, the element $w$ is not in
    $U^\perp$---for a single such $U$, this condition defines a Zariski-open and nonempty set, and there are only
    finitely many $U$ subordinate to $\sh A$.
  \end{enumerate}
  We may therefore choose some $w \in W$ meeting all the conditions listed above. We claim that the set $\sh A \cup
  \{w\}$ is very good.

  Suppose that $U, U^\perp$ are both subordinate to $\sh A \cup \{w\}$. We want to show that
  \[U \in \{0, \Iso(W), W, W^\perp, W+W^\perp, V\}.\]
  Since $\<{w, w} \neq 0$, the element $w$ can lie in at most one of
  $U$, $U^\perp$. If $w$ does not lie in either, they are both subordinate to the good set $\sh A$, and we can conclude.

  Without loss of generality, assume $w \in U^\perp$. It must be the case that $U^\perp \supseteq W$, by our choice of
  $w$. Since $U^\perp$ is subordinate to $\sh A \cup \{w \} \subset W$, it must be the case that $U^\perp = W$, and so
  $U = W^\perp$. This shows that $\sh A \cup \{w\}$ is good.

  We must show that it is very good. Suppose that $U$ and $U^\perp \cap W$ are both subordinate to $\sh A \cup \{w\}$, and $U \supseteq \Iso(W)$. Since $\sh A$ is very good by
  hypothesis, it must be the case that $w$ is in exactly one of $U$ and $U^\perp$. There are two cases.

  Suppose $w \in U$. Then consider
  $U + W^\perp$; since this contains $w$ and is orthogonal to $U^\perp \cap W$, which is subordinate to $\sh A$,
  point \ref{p3} above says that $U + W^\perp \supseteq W$. Since $U \subseteq W$, 
  we deduce that $U + W^\perp = W + W^\perp$. Since $U \supseteq \Iso(W)$ by hypothesis, it follows that $U = W$.

  Now suppose $w \in U^\perp \cap W$, so that $U$ is subordinate to $\sh A$. Point \ref{p3} says that $U^\perp \supseteq
  W$, hence $U \subseteq W^\perp$. We know $\Iso(W) \subseteq U$ by hypothesis, $U \subseteq W$ since $\Span \sh A
  \subseteq W$, and $U \subseteq W^\perp$. Therefore $U = \Iso(W)$.
  
  In either case, we deduce that $\sh A \cup \{w\}$ is very good. Proceeding inductively, we may extend $\sh B_0$ to a
  very good set $\sh B_0 \cup \sh B_1$ that is an ordered basis for $W$.

  \medskip

  Now we extend to a good basis for $W+ W^\perp$. Suppose $\sh A$ is a good set that contains $\sh B_0 \cup \sh B_1$ but
  that does not span $W+W^\perp$. The following conditions on $w \in W^\perp$ determine nonempty Zariski-open subsets of
  $W^\perp$.
  \begin{enumerate}
  \item $w \not \in \Span \sh A$,
  \item $\<{w, w} \neq  0$---this condition is open, and if it determines the empty set in $W^\perp$, then $W^\perp= W =
    \Iso(W)$, so $\sh B_{0}$ already spans $W+ W^\perp$.
  \item for all $U$ subordinate to $\sh A$ and for which $U^\perp \not \supseteq W^\perp$, the element $w$ is not in
    $U^\perp$---for a single such $U$, this condition defines a Zariski-open and nonempty set, and there are only
    finitely many $U$ subordinate to $\sh A$.
  \end{enumerate}
  We may therefore choose some $w \in W$ meeting all the conditions listed above. We claim that the set $\sh A \cup
  \{w\}$ is good.

  Suppose now that $U, U^\perp$ are both subordinate to $\sh A \cup \{w\}$. We will show that
 \[ U \in \{0, \Iso(W), W, W^\perp, W+W^\perp, V\}.\] Since $\<{w, w} \neq 0$, the element $w$ can lie in at most one of
  $U$, $U^\perp$. If $w$ does not lie in either, they are both subordinate to the good set $\sh A$, and we can conclude.

  Without loss of generality, therefore, assume $w \in U^\perp$. Then it must be the case that
  $U^\perp \supseteq W^\perp$, by our choice of $w$. We may write $U^\perp$ as $W^\perp + Y$ where $Y$ is subordinate to
  $\sh B_0 \cup \sh B_1$ and contains $\Iso(W)$, then $U = W \cap Y^\perp$, which is also subordinate to
  $\sh B_0 \cup \sh B_1$. Since $\sh B_0 \cup \sh B_1$ is a very good basis, this implies that $Y$ is either $\Iso(W)$,
  in which case $U^\perp = W^\perp$, or $Y=W$, in which case $U^\perp$ is $W+ W^\perp$.

  This shows we may enlarge $\sh B_0 \cup \sh B_1$ to $\sh B_0 \cup \sh B_1 \cup \sh B_2$, a good basis for $W+ W^\perp$
  and where $\sh B_0 \cup \sh B_2$ is a basis for $W^\perp$.

  \medskip

  Finally, we extend $\sh B_0 \cup \sh B_1 \cup \sh B_2$ to a good basis of $V$. Let $\sh A$ be a good subset of $V$
  containing $\sh B_0 \cup \sh B_1 \cup \sh B_2$ but that is not a basis for $V$. The following conditions on $w \in V$ determine nonempty Zariski-open
  subsets of $V$.
  \begin{enumerate}
  \item $w \not \in \Span \sh A$.
  \item $\<{w, w} \neq  0$.
  \item for all nonzero $U$ subordinate to $\sh A$, the element $w$ is not in
    $U^\perp$---for a single such $U$, this condition defines a Zariski-open and nonempty set, and there are only
    finitely many $U$ subordinate to $\sh A$.
  \end{enumerate}
  We may therefore choose some $w \in W$ meeting all the conditions listed above. We claim that the set $\sh A \cup
  \{w\}$ is good.

  Suppose for the sake of contradiction that $U$ and $U^\perp$ are both subordinate to $\sh A \cup \{w\}$ and do not lie
  in the set $\{0, \Iso(W), W, W^\perp, W \cap W^\perp \}$. Since $\sh A$ is good, at least one of $U$ and $U^\perp$
  must contain $w$. Without loss of generality, $w \in U^\perp$, and since $\<{w,w} \neq 0$, it must be the case that
  $w \not \in U$. Then $U$ must be subordinate to $\sh A$, contradicting the choice of $w$.

  Therefore we may extend $\sh B_0 \cup \sh B_1 \cup \sh B_2 $ to a good basis of $V$.  
\end{proof}

This lemma allows us the prove the following proposition.
\begin{proposition} \label{pr:tupleExists}
  Let $W \subseteq V$ be a subspace of rank $d$ and isotropy rank $l$. Let $r \ge 1$. There exists an $r$-tuple
  $A_\bullet$ in $Z(W, r; V) \cap Z(l,d, r; V)$ such that any subspace $U \subseteq V$ that is $\rho$-invariant under
  $A_\bullet$ is an element of
  \[ \{0, \Iso(W), W, W^\perp, W+ W^\perp, V\}. \]
  Furthermore, $A_1$ has only finitely many invariant subspaces.
\end{proposition}
\begin{proof}
  First we produce a suitable matrix $A_1$.
  
  It is possible to find a good basis $\sh B$ for $V$ in the sense of Lemma \ref{lem:longLemma}. For each $i \in \{1,
  \dots, n\}$, choose an element $a_i \in k$ such that $a_i = a_j$ only if $i=j$, then define $A_1: V \to V$ by
  $A(b_i) = a_i b_i$. Then $A_1$ has $n$ distinct eigenvalues, and by Lemma \ref{lem:nonderogatory}, the subspaces
  $U \subseteq V$ that are invariant under $A$ are precisely the spaces that are subordinate to $\sh B$. The spaces $U$
  that are invariant under $\rho(A_1)$ are the spaces for which $U^\perp$ is subordinate to $\sh B$. Only the spaces in
  $\{ \Iso(W), W, W^\perp, W+W^\perp ,V \}$ satisfy the condition that both $U$ and $U^\perp$ are subordinate to
  $\sh B$, and therefore these are the only spaces that are $\rho$-invariant under $A_1$.

  \medskip
  
  Now produce the $r$-tuple $(A_1, 0, \dots , 0)$. This has both $W$ and
  $W^\perp$ as invariant subspaces, so it lies in $Z(W, r; V)$. The subspaces $Y$ such that both $Y$ and $Y^\perp$ are
  invariant are those in the set $\{0 , \Iso(W), W, W^\perp, W+ W^\perp, V\}$. If any of these has the same dimension as
  $W$, then it also has the same isotropy rank as $W$. As a consequence, $A_\bullet $ is in $Z(\ge l
  , d , r; V)$ but not in $Z(\ge l+1 , d, r; V)$, so $A_\bullet$ is in $Z(l, d,r; V)$.
\end{proof}

\subsection{Irreducible components of \texorpdfstring{$Z(l,d,r; V)$}{Z(l,d,r;V)}}

\label{ss:irreducibleCompsOfZ(ldrV)}

\begin{lemma} \label{lem:actionIrred}
 Let $K_0$ be a topological group that is irreducible as a topological space. Let $X$ be a space with a $K_0$ action,
 and let $Y \subseteq X$ be an irreducible component. Then the action of $K_0$ on $X$ restricts to an action of $K_0$
 on $Y$.
\end{lemma}
\begin{proof}
 The action map restricts to a map $\alpha: K_0 \times Y \to X$, and the source of this map is irreducible. Therefore, so too
 is the image, but this contains $Y$, an irreducible component. Therefore the image of $\alpha$ is precisely $Y$.
\end{proof}

Recall that $G$ denotes either $\Og(V)$ or $\Sp(V)$, depending on whether the bilinear form on $V$ is symmetric or
skew-symmetric. The group $PG$ is the quotient of $G$ by its centre, which in each case consists of $\pm 1$ times the identity. It is
well known that $\Sp(V)$, and therefore $\PSp(V)$, is connected \cite[Section 2.2.2]{Springer1994}. It is also well known
that $\Og(V)$ consists of two components, one being $\SO(V)$ and the other consisting of matrices of determinant $-1$. One sees
directly that if $n$ is odd, then $\PO(V)$ is connected, since the quotient map $\Og(V) \to \PO(V)$ identifies the two
components in this case. On the other hand, $\PO(V)$ consists of two connected components when $n$ is even.

If $K$ is a topological group, then $K_0$, the connected component of the identity, is a
normal subgroup and the quotient group $K/K_0$ is naturally isomorphic to the set of connected components
$\pi_0(K)$. The astute reader will observe that $\pi_0(K)$ usually denotes the set of path components, but the spaces we
consider are finite disjoint unions of irreducible varieties, and are therefore locally path connected, so the abuse of
notation will have no practical effect.

\begin{lemma} \label{lem:ZEqualComponents}
  Let $1 \le d \le n/2$ and let $l \le d$. Let $K \subseteq PG$ be an algebraic
  subgroup that acts transitively on $\Gr(l,d;V)$. Then
  $\pi_0(K)$ acts transitively on the set of irreducible components of $Z(l,d, r;V)$.
\end{lemma}

\begin{proof}
  To avoid vacuity, we assume that $Z(l,d,r; V)$ is not empty. The group $G$, and therefore $K$, acts on the set of
  irreducible components of $Z(l , d, r ;V)$. We claim that the action of $K$ on this set is transitive. 

  We rely on the observation that if $A_\bullet$ is a nonsingular point of the variety $Z(l,d,r;V)$, then $A_\bullet$
  lies in exactly one irreducible component of $Z(l,d,r;V)$.

  Suppose $X, X'$ are irreducible components of $Z(l,d,r;V)$. Choose points $A_\bullet \in X$ and
  $A'_\bullet \in X'$ that are both nonsingular in $Z(l, d, r;V)$. We will show there is an element $g \in K$
  such that $g A_\bullet \in X'$. To produce $g$, find $\rho$-invariant
  $d$-dimensional subspaces $W$ and $W'$ for $A_\bullet$ and $A'_\bullet$, each of isotropy rank $l$. There exists some
  $g \in K$ taking $W$ to $W'$, since $g$ acts transitively on $\Gr(l, d; V)$. Then $g A_\bullet$ is a nonsingular point in
  $Z(l, d, r; V)$ and also lies in $Z(W', r; V)$.

  Define $Y=Z(W', r; V) \cap Z(l, d, r;V)$. We know that $Z(W', r; V) \subset Z(\ge l , d, r; V)$, so that $Y$ is open
  in $Z(W', r; V)$, and therefore irreducible (see Proposition \ref{pr:dimOfZWr}). Now $gA_\bullet$ and $A'_\bullet$ are
  two points in $Y$ and both nonsingular in $Z(l, d, r;V)$, and are contained in the same irreducible subset,
  $Y$. Therefore they determine the same irreducible component of $Z(l,d,r;V)$, which proves that $K$ acts transitively
  on the set of irreducible components.

  Since the action of $K_0$ on irreducible components of $Z(l,d,r;V)$ is trivial by Lemma \ref{lem:actionIrred}, there
  is an induced transitive action of $\pi_0(K)$ on the set of irreducible components, as required.
\end{proof}

\begin{proposition}\label{pr:irreducibility}
  Suppose the variety $Z(l,d,r; V)$ is not empty. Then $Z(l,d,r;V)$ is irreducible
  unless $\<{\cdot, \cdot}$ is symmetric, $2d=n$ and $l=d$. In the exceptional case, $Z(l,d,r;V)$ consists of exactly two
  irreducible components.
\end{proposition}
\begin{proof}
  We may assume $2d \le n$, since $Z(l,d,r;V) = Z(l, n-d, r;V)$.
  
  First we dispense with the non-exceptional cases. Two of these are direct. If $\<{\cdot, \cdot}$ is skew-symmetric, then applying Lemma
  \ref{lem:ZEqualComponents} using the connected group $\PSp(V)$ gives the result. If $\<{\cdot, \cdot}$ is symmetric but
  $n$ is odd, then applying the lemma using the connected group $\PO(V)=\SO(V)$ gives the result.

  Suppose now that $\<{\cdot, \cdot}$ is skew-symmetric and $n$ is even. If $W$ and $W'$ are two subspaces of $V$ of rank $d$ and isotropy rank $l$ then there exists an abstract
  isomorphism of quadratic spaces $f: W \to W'$. Witt's Theorem \cite[I.4.1]{Chevalley1954} says that $f$ may be
  extended to an element $F \in \Og(V)$. If $F$ has determinant $1$, there is nothing further to be done.

  If $F$ has determinant $-1$, then argue as follows. If $2d < n$ or $l <
  d$, then we may take some element $a \in W^\perp \sm (W^\perp \cap W)$, since
  $W^\perp \subseteq W$ would imply either $d > n-d$ or $l=d$. Precompose $F$ by the orthogonal transformation $S$ that
  acts as $-1$ on $\langle a \rangle$ and as the identity on $\langle a \rangle^\perp$. The composite $FS$ takes $W$ to
  $W'$ and has determinant $1$. Therefore, unless $2d=n$ and $l=d$, we may apply Lemma \ref{lem:ZEqualComponents} using
  the connected group $\operatorname{PSO}(V)$.

  In the last case, when the form is symmetric, $2d=n$ and $l=d$, there can be at most two irreducible components, by
  virtue of Lemma \ref{lem:ZEqualComponents}. It is well known that there are exactly two isomorphic orbits for the
  $\SO(V)$ action on the isotropic Grassmannian $\Gr(d,d;V)$ when $d = n/2$:
  \cite[p.~365]{Procesi2007}. This implies that $\Gr(d,d;V)$ is disconnected, consisting of two isomorphic components,
  since the closure of an orbit must consist of the union of that orbit and lower-dimensional orbits. Let us denote
  these components by $\Gr_a$ and $\Gr_b$, and let $W_a$ denote a subspace in $\Gr_a$. Consider the incidence variety
  $X \subseteq Z(d,d, r;V) \times \Gr(d,d;V)$ consisting of pairs $(A_\bullet, W)$ where $W$ is $\rho$-invariant
  under $A$. This too decomposes into disjoint closed subspaces: $X_a$ and $X_b$. Since the map
  $Z(d,d,r;V) \times \Gr(d,d ; V)$ is closed, the images of $X_a$ and $X_b$ in $Z(d,d,r;V)$ are closed. Denote these
  images by $Z_a$ and $Z_b$. Using Proposition \ref{pr:tupleExists}, we can produce an $r$-tuple lying in $Z_a$ but not
  in $Z_b$, for instance, one having only $W_a$ as a $d$-dimensional $\rho$-invariant subspace. Because $Z_a$ and $Z_b$ are
  interchanged by an automorphism of $Z(d,d,r; V)$, it follows that $Z(d,d,r;V) = Z_a \cup Z_b$ but neither of these
  closed sets is contained in the other. Therefore, $Z_a$ and $Z_b$ must be two distinct irreducible components of
  $Z(d,d,r;V)$.
\end{proof}

\section{Calculating dimensions} \label{sec:DimensionCalculation}

We now turn to calculating the dimension of $Z(l, d, r; V)$. The idea is simple enough: one expects
\begin{equation}
 \label{eq:1}
 \dim Z(l,d,r; V) = \dim \Gr(l,d; V) + \dim Z(W,r; V)
\end{equation}
where $W$ is some fixed subspace $W \subseteq V$ of dimension $d$ and isotropy rank $l$. On the right hand side, the
first summand is the dimension of the space of possible $\rho$-invariant subspaces and the second is the dimension of the $r$-tuples
associated with a given subspace.

The first summand was calculated in Proposition \ref{prop:gr-dim-sym} or Proposition \ref{pr:dimAltGr}, and the second
is calculated in \ref{pr:dimOfZWr}. The work is in justifying \eqref{eq:1}. The main difficulty is there is no map
$Z(l,d,r; V) \to \Gr(l,d;V)$ sending $A_\bullet$ to its associated $\rho$-invariant subspace, since an $r$-tuple
$A_\bullet$ may admit more than one $d$-dimensional $\rho$-invariant subspace $W$. What is true is that there is an open
subvariety $Y$ of $Z(l,d,r;V)$ consisting of $r$-tuples admitting only finitely many such subspaces, and so there is a
finite correspondence $Y \to \Gr(l,d;V)$. This is established in Corollary \ref{cor:Corresp}.

\subsection{Dimension as a sum}

\begin{notation} \label{not:incVar}
 Fix $d, l$ and $r$. Let $X \subset Z(l,d,r; V) \times \Gr(l,d; V)$ be the incidence variety consisting of pairs
 $(A_\bullet, W)$ where $A_\bullet$ is an $r$-tuple of endomorphisms having $W$ as a $\rho$-invariant
 subspace. There are obvious projection maps $X \to Z(l,d,r; V)$ and $X \to\Gr(l,d;V)$.
\end{notation}

\begin{proposition}
 There exists a dense open subset $U \subseteq Z(l,d,r; V)$ such that the projection \[ U \times_{Z(l,d,r; V)} X \to
   U \]
 is quasifinite.
\end{proposition}
\begin{proof}
  If $Z(l,d,r; V)$ is empty, there is nothing to show, so assume that there exists some $A_\bullet \in Z(l, d, r; V)$.
  Let $U$ be the open subset of $Z(l,d,r;V)$ consisting of tuples $A_\bullet$ for which $A_1$ has $n$ distinct
  eigenvalues. The matrix $A_1$ has finitely many invariant subspaces. The subset $U$ is dense, by Proposition
  \ref{pr:tupleExists}. Any point in the fibre of the map $X \to Z(l,d, r; V)$ over $A_\bullet$ must consist of a pair
  $(A_\bullet, Y)$ where $Y$ is invariant under $A_1$, and therefore there can be only finitely many.
\end{proof}

\begin{notation}
 Write $X_U$ for $U \times_{Z(l,d,r; V)} X$.
\end{notation}

\begin{corollary} \label{cor:Corresp}
 The dimension $\dim Z(l, d, r; V)$ agrees with $\dim X_U$.
\end{corollary}
\begin{proof}
 Since $U$ is dense and open in $Z(l, d, r; V)$, the dimensions $\dim U$ and $\dim Z(l, d, r; V)$ agree. The map
 $X_U \to U$ is a quasifinite map of varieties. By use of Zariski's main theorem, \cite[Th\'eor\`eme 8.12.6]{Grothendieck1966a} there is a dense open $Y \subseteq U$ such that $X_Y \to Y$ is a
 finite map, and therefore $\dim X_Y = \dim Y = \dim U$.
\end{proof}

\begin{theorem} \label{th:mainDimCalc}
 Let $n \ge 2$, let $1 \le d \le n/2$ and let $l \le d$. In the skew-symmetric case, suppose also that $l
 \equiv d \pmod 2$. Let $r \ge 1$. Then
 \begin{align} \dim Z(l,d,r; V) &= d(n-d) - \frac{ l^2 + l}{2} + r( (n-d)^2 + d^2 + l^2) \quad \text{ if $\<{\cdot,
 \cdot}$ is symmetric} \\
 \dim Z(l,d,r; V) &= d(n-d) - \frac{ l^2 - l}{2} + r( (n-d)^2 + d^2 + l^2) \quad \text{ if $\<{\cdot, \cdot}$ is skew-symmetric} \end{align}
\end{theorem}
\begin{proof}
  We know that $\dim Z(l, d, r ; V) = \dim X_U$. Here $X_U$ is the subvariety of the incidence variety defined in
  Notation \ref{not:incVar} consisting of pairs $(A_\bullet, W)$ where $A_\bullet$ is an $r$-tuple such that $A_1$ has
  $n$ distinct eigenvalues. The group $G$ acts diagonally on $X_U$, by conjugation on the endomorphisms and by the usual
  action on the subspaces $W$. The group $G$ acts transitively on $\Gr(l,d; V)$ and moreover the map
  $\pi: X_U \to \Gr(l,d; V)$ is equivariant. Since the map $X_U \to \Gr(l, d; V)$ is generically flat \cite[Th\'eor\`eme
  6.9.1]{Grothendieck1965}, and flatness is local, it is flat everywhere. Consequently
 \begin{equation}
 \label{eq:3}
 \dim X_U = \dim \Gr(l, d; V) + \dim F
 \end{equation}
 by \cite[Proposition III.9.5]{HartshorneAlgebraicGeometry1977}, where $F$
 denotes the fibre of $\pi$ over some fixed $W \in \Gr(l,d;V)$.

 The fibre $F$ is identified with an open subvariety of $Z(W, r; V)$ consisting of $r$-tuples $A_\bullet = (A_1, \dots, A_r)$ where
 $A_1$ has $n$ distinct eigenvalues, and such that the tuple $A_\bullet$ does not have any $\rho$-invariant subspace $Y$ of isotropy rank greater than $l$. The variety $Z(W, r; V)$ is irreducible, by Proposition \ref{pr:dimOfZWr}, and $F$
 is nonempty by Proposition \ref{pr:tupleExists}. The dimension of $F$ is therefore also
 \[ \dim F= r((n-d)^2 + d^2 + l^2) \]
 by Proposition \ref{pr:dimOfZWr}.

 We established the identities
 \[ \dim \Gr(l,d;V) = \begin{cases} d(n-d) - \frac{ l^2 + l}{2} \quad \text{in the symmetric case} \\
 d(n-d) - \frac{ l^2 - l}{2} \quad \text{in the skew-symmetric case} \end{cases} \]
 in Propositions \ref{prop:gr-dim-sym} and \ref{pr:dimAltGr}. Combining the dimension calculations using \eqref{eq:3}
 yields the result.
\end{proof}

\subsection{The irreducible components of \texorpdfstring{$Z(r;V)$}{Z(r;V)}}

Fixing $V$, and therefore $n = \dim(V)$, and $r$, we determine the irreducible components of $Z(r; V)$ and identify the
components of maximal dimension.

\begin{lemma} \label{pr:dstrdec}
  If we keep $V$, $r \ge 1$ and $l \ge 0$ fixed, where $2l \le n$, and allow \[d \in \{ \max\{l, 1\}, \max\{l, 1\}+1,
  \dots, [n/2]\},\] then the quantity
  $\dim Z(l,d,r; V)$ is a strictly decreasing function of $d$.
\end{lemma}
\begin{proof}
  Treating $\dim Z(l,d,r; V)$, calculated in Theorem \ref{th:mainDimCalc}, as a function of one real variable $d$  and
  differentiating, we obtain the derivative
  \[ (1-2r)(n-2d). \]
  This is negative on the interval $(l,n/2)$, and consequently $\dim Z(l,d,r; V)$ is strictly decreasing on the closed
  interval $[l, n/2]$.
\end{proof}

One already knows from Proposition \ref{pr:inclusionInZll} that $Z(l,d,r; V) \subseteq Z(l,l,r;V)$, so that the
dimension was known to be weakly decreasing. We have now established:
\begin{corollary}
  The inclusion $\overline{Z(l,d, r ; V)} \subseteq Z(l,l,r;V)$ is the inclusion of a closed subvariety of strictly
  smaller dimension whenever $l < d$.
\end{corollary}

We now know enough to determine the irreducible components of $Z(r;V)$, the subvariety of $\End_k(V)^r$ consisting of
$r$-tuples that do not generate $\End_k(V)$ as an algebra with involution. We begin with a small lemma.

\begin{lemma} \label{lem:irredSubsFinite}
  Suppose $X$ is a topological space that admits a decomposition
  \[ X = \bigcup_{i=1}^n A_i \]
  where each $A_i$ is an irreducible closed subset of $X$. The $A_i$ are the irreducible components of $X$ if and only
  if no $A_i$ is contained in any other.
\end{lemma}
\begin{proof}
  If the $A_i$ are irreducible components, then by definition no $A_i$ can be contained in $A_j$ when $i \neq j$.

  In the other direction, we first observe that if $M$ is an irreducible component of $X$, then
  \[ M = \bigcup_{i=1}^n (A_i \cap M) \] so that $M$ must actually equal one of the $A_i$. Therefore the irreducible
  components of $X$ are drawn from the set $\{A_1, \dots, A_n\}$. If some $A_i$ in this set were not an irreducible
  component, it would be contained in some other $A_j$, contrary to hypothesis.
\end{proof}

\begin{theorem} \label{th:closedCoverOfZsymm}
  Define
  \[ I = \begin{cases} \{ d \in \ZZ \mid 1 \le d, \quad 2d \le n \} \quad & \text{ if $\<{\cdot, \cdot}$ is
        symmetric. }\\ \{ 
      d \in \ZZ \mid 1 \le d,\; 2d \le n, \; d\equiv 0 \!\!\!\pmod{2} \} \quad & \text{ if $\<{\cdot, \cdot}$ is
        skew-symmetric. }\end{cases} \] The variety $Z(r;V)$ admits a cover by closed subvarieties
  \begin{equation}
    Z(r; V) = \bigcup_{d=1}^{[n/2]} Z(d,d,r;V) \cup \bigcup_{d\in I} \overline{ Z(0, d, r ;V) }. \label{eq:7}
  \end{equation}
  The irreducible components of the varieties appearing in \eqref{eq:7} are the irreducible components of $Z(r;V)$.
\end{theorem}

  According to Proposition \ref{pr:irreducibility}, each closed subvariety appearing in the cover in \eqref{eq:7} is
  irreducible, with the exception of $Z(d,d,r;V)$ when $2d = n$, in which case there are two
  irreducible components.

\begin{proof}
  There is a cover of $Z(r; V)$ by the subsets $Z(d,r;V)$ as $d$ ranges from $1$ to $n-1$. There
  are equalities of subsets $Z(d,r;V)=Z(n-d,r;V)$, so we consider only $d$ such that $2d \le n$.

  The subsets $Z(d,r;V)$ of $Z(r,V)$ are covered by locally closed subvarieties of $\End_k(V)^r$:
  \[ Z(d,r;V) = \bigcup_{l=0}^d Z(l,d,r;V), \]
  and we know that if $l \ge 1$, then $Z(l,d,r;V) \subseteq Z(l,l,r;V)$ (Proposition~\ref{pr:inclusionInZll}). Therefore, the
  union
  \[ \bigcup_{d=1}^{[n/2]} Z(d,d,r;V) \cup \bigcup_{d\in I} \overline{ Z(0, d, r ;V)} \]
  is sufficient to cover $Z(r;V)$. Moreover, $Z(d,d,r;V) =Z(\ge \! d,d,r;V)$ is a closed subvariety of $\End_k(V)^r$,
  and so the cover in \eqref{eq:7} is a cover by closed subvarieties.

  To show that the irreducible components of the varieties appearing in \eqref{eq:7} are the irreducible components of
  $Z(r;V)$, we use Lemma \ref{lem:irredSubsFinite}. Since each variety appearing consists of one or two irreducible
  components, it is sufficient to establish that no irreducible component of any of the closed varieties on the right
  hand side of \eqref{eq:7} is contained in any other. We prove this by checking cases.

  No irreducible component of a variety $Z(d,d,r; V)$ is contained in $Z(d' , r ; V)=Z(\ge 0 ,d' ,r; V)$ if $d'\neq
  d$. This can be seen from Proposition \ref{pr:tupleExists}, since for any $d \in \{1, \dots, [n/2]\}$, and any
  component of $Z(d,d,r;V)$, there exists an $r$-tuple $A_\bullet$ in that component and a $d$-dimensional totally
  isotropic subspace $W$ such that $W$ is the only proper $\rho$-invariant subspace of $A$. Consequently, no component
  of $Z(d,d,r;V)$ is contained in $Z(d',d', r; V) \cup \overline{Z(0, d',r;V)}$ when $d' \neq d$.
  
  The dimension of $Z(d,d,r ;V)$ is not less than that of $Z(0,d,r;V)$, by Theorem \ref{th:mainDimCalc}. This implies
  that no component of $Z(d,d,r;V)$ is contained in the irreducible closed variety $\overline{Z(0,d,r;V)}$.
    
  The variety $Z(0,d,r;V)$ is not contained in $Z(d',d',r;V)$ for any value of $d \in I$ and
  $d' \in \{1, \dots, [n/2]\}$. This can be seen from Proposition \ref{pr:tupleExists}.
 
  For any $(d,d') \in I^2$ where $d \neq d'$, the variety $Z(0,d,r;V)$ is not contained in the closed variety
  $Z(d' ,r ;V)$. This can be seen again from Proposition \ref{pr:tupleExists}, because there exists an $r$-tuple
  $A_\bullet$ and a $d$-dimensional anisotropic subspace $W$ such that the proper $\rho$-invariant subspaces for $A$ are
  $W$ and $W^\perp$. If $\dim W^\perp \in I$, it must be the case that $n-d \ge n/2$, whereupon $d = n/2 = n-d$.

  Since $\overline{Z(0, d', r; V)}$ is contained in the closed set $Z(d',r;V)$, this establishes that
  $\overline{Z(0,d,r, V)}$ is not contained in $\overline{Z(0,d',r,V)}$ when $d \neq d'$.
\end{proof}

\begin{proposition} \label{pr:dimSymmNot4} Suppose $\<{\cdot, \cdot}$ is symmetric and that $n =\dim_k V \ge 2$. For a
  given $r$, the maximal dimension of $Z(l,d,r;V)$ is attained by and only by
 \begin{enumerate}
 \item $\dim Z(0,1,1;V)=\dim Z(1,1,1;V)=\dim Z(2,2,1;V) = 13$ if $n=4$ and $r=1$;
 \item $\dim Z(2,2,r;V) = 12r +1$ if $n=4$ and $r \ge 2$;
 \item $\dim Z(0,1,1;V) = \dim Z(1,1,1;V) =n^2-n+1$ if $n \neq 4$ and $r=1$;
 \item $\dim Z(1,1,r;V) = n-2 + r(n^2-2n+3)$ if $n \neq 4$ and $r \ge 2$.
 \end{enumerate}
\end{proposition}
\begin{proof}
  There is no loss of generality to restricting to $d \le n/2$, since $Z(l,d,r;V) = Z(l,n-d,r;V)$.
  
  In light of Lemma \ref{pr:dstrdec}, the maximal dimension among all possible $Z(l,d,r;V)$ can be attained
  only when $d=l$, or when $l=0$ and $d=1$.

  When $d=l$, the calculation of Theorem \ref{th:mainDimCalc} gives
  \[ \dim Z(d,d,r;V) = d(n-d) -\frac{d^2+d}{2} + r ((n-d)^2 + 2d^2 ) \]
  which is a quadratic function of $d$. By direct calculation, the leading coefficient is $(3r - 3/2)$, which is positive, and
  so the maximal dimension is attained at one or other endpoint of the interval under consideration, i.e., when $d=1$ or
  $d=[n/2]$. A tedious, but elementary, verification shows that $\dim_k Z(1,1,r; V) > \dim_k([n/2], [n/2], r; V)$ for all values of
  $r$ when $n \ge 5$. This yields the maximum value:
  \begin{equation}
    \dim Z(1,1,r; V) = (n-2) + r(n^2-2n+3) \quad \text{provided $n\neq 4$.}\label{eq:8}
  \end{equation}

  When $n=2,3$, there is only one value of $d$ to consider, namely $d=1$. When $n=4$, we calculate
  \[ \dim Z(2,2,r;V) = 12r +1, \quad \dim Z(1,1,r; V) = 11r + 2. \]
  
  When $l =0$ and $d=1$, the calculation of Theorem \ref{th:mainDimCalc} says
  \begin{equation}
    \dim Z(0,1,r;V) =  (n-1)  + r n^2 - 2rn + 2r\label{eq:9}
  \end{equation}
  which is strictly less than \eqref{eq:8} unless $r=1$, whereupon they are equal.

  Combining all the above gives the result.
\end{proof}

\begin{remark}
  We remark that in the case of an orthogonal involution, the maximal dimensions of the irreducible components are
  \begin{enumerate}
  \item $12r +1$ if $n=4$;
 \item $r(n^2-2n+3) + n -2$ if $n \neq 4$.
  \end{enumerate}
\end{remark}


Recall that $U(r;V) \subseteq \End_k(V)$ was defined to be the set of $r$-tuples of endomorphisms that generate
$\End_k(V)$ as an algebra with involution.

\begin{corollary} \label{cor:SymmDenseOpen}
 Suppose $\<{\cdot, \cdot}$ is symmetric and $r \ge 1$. Then $U(r;V) \subseteq \End_k(V)^r$ is a dense open subvariety.
\end{corollary}
\begin{proof}
 Since $U(r;V)$ is the complement of the closed subvariety $\bigcup_{d \le [n/2]} Z(\ge l,d,r; V)$, it is
 open. The maximal dimension of
 \[ Z(\ge l, d, r; V) = \bigcup_{l=0}^{d} Z(l , d,r ;V) \]
 as $d$ ranges from $1$ to $[n/2]$ and $l$ from $0$ to $d$ is given in Proposition \ref{pr:dimSymmNot4}, and in each
 case is strictly less than $rn^2 = \dim \End_k(V)$.
\end{proof}

\begin{proposition} \label{pr:dimSkew-Symmetric}
 Suppose $n = \dim V$ is even, and $\phi$ is skew-symmetric. The maximal dimension of $Z(l,d,r;V)$ is attained
 by and only by
 \begin{enumerate}
 \item $\dim Z(2,2,r; V) = 12r+3$ when $n=4$;
 \item $\dim Z(3,3,r; V) = 27r + 6$ when $n=6$;
 \item $\dim Z(4,4, 1; V)=\dim Z(1,1,1;V) = 58$ when $n=8$ and $r=1$;
 \item $\dim Z(1,1,r;V) = (n-1)^2r +2r+n-1$ in all other cases.
 \end{enumerate}
\end{proposition}

\begin{proof}
  There is no loss of generality to restricting to $d \le n/2$, since $Z(l,d,r;V) = Z(l,n-d,r;V)$.
  
  In light of Lemma \ref{pr:inclusionInZll}, the maximal dimension among all possible $Z(l,d,r;V)$ can be attained
  only when $d=l$, or when $l=0$ and $d=1$. Since $\<{\cdot,\cdot}$ is skew-symmetric, the case $l=0$ and $d=1$ cannot
  arise. 

  When $d=l$, the calculation of Theorem \ref{th:mainDimCalc} gives
  \[ \dim Z(d,d,r;V) = d(n-d) -\frac{d^2-d}{2} + r ((n-d)^2 + 2d^2 ) \]
  which is a quadratic function of $d$. By direct calculation, the leading coefficient is $(3r - 3/2)$, which is positive, and
  so the maximal dimension is attained at one or other endpoint of the interval under consideration, i.e., when $d=1$ or
  $d=[n/2]$. Elementary verification shows that $\dim_k Z(1,1,r; V) > \dim_k([n/2], [n/2], r; V)$ for all values of
  $r$ when $n \ge 10$. When $n=2$, there is only one case, that of $d=1$. The remaining three cases, $n=4,6,8$, are straightforward.
\end{proof}


\begin{corollary} \label{cor:baldDimensionCountsSkew}
  Suppose the bilinear form on $V$ is skew-symmetric. The maximal
  dimension of an irreducible component of $Z(r;V)$ is
  \begin{enumerate}
  \item $12r +3$ if $n=4$;
    \item $27r + 6$ if $n=6$;
 \item $n^2 r - 2nr + 3r + n - 1$ if $n =2$ or $n \ge 8$.
 \end{enumerate}
\end{corollary}

\begin{corollary} \label{cor:SkewSymmDenseOpen}
 Suppose $\<{\cdot, \cdot}$ is symmetric and $r \ge 1$. Then $U(r;V) \subseteq \End_k(V)^r$ is a dense open subvariety
 unless $n =2$ and $r=1$.
\end{corollary}
\begin{proof}
 The argument is essentially the same as in Corollary \ref{cor:SymmDenseOpen}.
\end{proof}

\begin{remark}
  In the case of $n=2$ and a skew-symmetric bilinear form, it is possible to write (up to algebra isomorphism)
  $\rho(A) = \Tr(A) I_2 - A$, so that the subalgebra with involution generated by any single element is
  commutative. Therefore $\End_k(V)$ cannot be generated by one element in this case.
\end{remark}

\emergencystretch=1em
\printbibliography

\end{document}
